\documentclass[final]{dmtcs-episciences}

\usepackage[utf8]{inputenc}
\usepackage{subcaption}

\usepackage[round]{natbib}

\usepackage{amsmath}
\usepackage{amssymb}
\usepackage{amsthm}
\usepackage{amsfonts}
\DeclareMathSymbol{\shortminus}{\mathbin}{AMSa}{"39} 

\usepackage{cancel}
\usepackage{multirow}
\usepackage{makecell}
\usepackage{mathtools} 
\usepackage{graphbox} 
\usepackage[table,dvipsnames]{xcolor}

\hypersetup{
  colorlinks  = true, %
  urlcolor    = black, 
  linkcolor   = black, 
  citecolor   = black  
}

\theoremstyle{plain}
\newtheorem{lem}{Lemma}
\newtheorem{cor}{Corollary}
\newtheorem{theorem}{Theorem}

\newtheorem{con}{Conjecture}

\newtheorem{prp}{Proposition}
\newtheorem{rem}{Remark}
\newtheorem{definition}{Definition}

\newcommand{\avoidSymbol}{\mathsf{Av}}

\newcommand{\avoid}[3][]{\avoidSymbol^{#1}_{#2}(#3)}

\newcommand{\vincular}[1]{\overline{#1}}

\newcommand{\SSymbol}{S}
\newcommand{\PERMS}[2][]{\SSymbol^{#1}_{#2}}
\newcommand{\MSymbol}{M}
\newcommand{\MULTIPATS}[2][]{\MSymbol^{#1}_{#2}}

\newcommand{\CatNum}[1]{\mathcal{C}_{#1}}
\newcommand{\KatNum}[2]{\mathcal{C}_{#1,#2}}

\newcommand{\prefixTree}[1]{T_{#1}}

\newcommand{\redText}[1]{\textcolor{BrickRed}{#1}}
\newcommand{\blueText}[1]{\textcolor{BlueViolet}{#1}}
\newcommand{\greenText}[1]{\textcolor{PineGreen}{#1}}

\newcommand{\OEIS}[1]{\href{https://oeis.org/#1}{#1}}

\graphicspath{{graphics/}}

\author[Downing et al.]
{
Emily Downing\affiliationmark{1}
\and Elizabeth J. Hartung\affiliationmark{1}
\and Cody Lucido\affiliationmark{1}
\and Aaron Williams\affiliationmark{2}
}

\affiliation{
  Massachusetts College of Liberal Arts, North Adams, USA\\
  Williams College, Williamstown, USA}

\title{Pattern Avoidance for Fibonacci Sequences using $k$-Regular Words}


\begin{document}

\publicationdata{vol. 26:1, Permutation Patterns 2023}{2025}{11}{10.46298/dmtcs.12752}{2023-12-28; 2023-12-28; 2025-01-17; 2025-05-14}{2025-05-15}

\maketitle

\begin{abstract}
\vspace{10pt}
Two $k$-ary Fibonacci recurrences are $a_k(n) = a_k(n-1) + k \cdot a_k(n-2)$ and $b_k(n) = k \cdot b_k(n-1) + b_k(n-2)$.
We provide a simple and direct proof that $a_k(n)$ is the number of $k$-regular words over $[n] = \{1,2,\ldots,n\}$ that avoid patterns $\{121, 123, 132, 213\}$ when using base cases $a_k(0) = a_k(1) = 1$ for any $k \geq 1$.
This was previously proven by Kuba and Panholzer in the context of Wilf-equivalence for restricted Stirling permutations, and it creates Simion and Schmidt's classic result on the Fibonacci sequence when $k=1$, and the Jacobsthal sequence when $k=2$.
We complement this theorem by proving that $b_k(n)$ is the number of $k$-regular words over $[n]$ that avoid $\{122, 213\}$ with $b_k(0) = b_k(1) = 1$ for any~$k \geq 2$.
Finally, we prove that $|\avoid[2]{n}{\vincular{121}, 123, 132, 213}| = a_1(n)^2$ for $n \geq 0$.
That is, vincularizing the Stirling pattern in Kuba and Panholzer's Jacobsthal result gives the Fibonacci-squared numbers.
\end{abstract}


\section{Introduction}
\label{sec:intro}

The Fibonacci sequence is arguably the most famous integer sequence in mathematics, and the term \emph{generalized Fibonacci sequence} has been used to describe an increasingly wide variety of related sequences.  
Here we consider two families of generalizations involving a second parameter $k$.

\begin{definition} \label{def:Fibk}
The \emph{Fibonacci-$k$ numbers} are $a_k(n) = a_k(n-1) + k \cdot a_k(n-2)$ with $a_k(0) = a_k(1) = 1$.
\end{definition}

\begin{definition} \label{def:kFib}
The \emph{$k$-Fibonacci numbers} are $b_k(n) = k \cdot b_k(n-1) + b_k(n-2)$ with $b_k(0) = b_k(1) = 1$.
\end{definition}


We provide pattern avoidance results for the sequences $\{a_k(n)\}_{n \geq 0}$ and $\{b_k(n)\}_{n \geq 0}$.
The objects are \emph{$k$-regular words} meaning that each symbol in $[n] = \{1,2,\ldots,n\}$ appears $k$ times, and  no subword can have the same relative order as any of the patterns.
We let $\PERMS[k]{n}$ be the set of $k$-regular words over $[n]$ and $\avoid[k]{n}{\pi_1, \pi_2, \ldots, \pi_m} \subseteq \PERMS[k]{n}$ be the subset that avoids all $m$ patterns.
We focus on two families of~words.

\begin{definition} \label{def:FibWordsk}
The \emph{Fibonacci-$k$ words} of length $kn$ are the words in $\avoid[k]{n}{121, 123, 132, 213} \subseteq \PERMS[k]{n}$.
\end{definition}

\begin{definition} \label{def:kFibWords}
The \emph{$k$-Fibonacci words} of length $kn$ are the words in $\avoid[k]{n}{122, 213} \subseteq \PERMS[k]{n}$.
\end{definition}

We provide a simple proof that the Fibonacci-$k$ words are counted by the Fibonacci-$k$ numbers for~\mbox{$k \geq 1$}.
This was previously proven by Kuba and Panholzer (see the $\overline{C_2}$ case of Theorem 3 in the citation).


\begin{theorem}[\cite{kuba2012enumeration}] \label{thm:a}
$a_k(n) = |\avoid[k]{n}{121,123,132,213}|$ for all $k \geq 1$ and $n \geq 0$.
\end{theorem}

For example, the Fibonacci-$2$ numbers $a_2(n)$ create \textsc{Oeis}\footnote{See the Online Encyclopedia of Integer Sequences (\textsc{Oeis}) \cite{OEIS} for all sequence references.} \OEIS{A001045} given by $1, 1, 3, 5, 11, 21, 43,  \ldots$,
and also known as the \emph{Jacobsthal sequence}.
Its first four terms count the Fibonacci-$2$ words in \eqref{eq:a20}--\eqref{eq:a23}.
\begin{align}
1 &= |\avoid[2]{0}{121,123,132,213}| = |\{ \epsilon \}| \label{eq:a20} \\ 
1 &= |\avoid[2]{1}{121,123,132,213}| = |\{11\}| \label{eq:a21} \\
3 &= |\avoid[2]{2}{121,123,132,213}| = |\{1122, 2112, 2211\}| \label{eq:a22} \\  
5 &= |\avoid[2]{3}{121,123,132,213}| = |\{223311, 322311, 331122, 332112, 332211\}| \label{eq:a23}
\end{align}
Note that $233211 \in \PERMS[2]{3}$ since it is a $2$-regular word over $[3]$.
However, $\underline{23}3\underline{2}11 \notin \avoid[2]{3}{121,123,132,213}$ since its underlined subword $232$ is order isomorphic to the pattern $121$.
Hence, it does not appear in~\eqref{eq:a23}. 

Kuba and Panholzer previously proved Theorem \ref{thm:a} in a much broader context involving Stirling permutations (see Section \ref{sec:intro_Stirling}).
Here we provide a simple and direct combinatorial proof.
We also complement their result for Fibonacci-$k$ words with a new pattern avoidance result for $k$-Fibonacci words.
More specifically, we prove that the $k$-Fibonacci words are counted by the $k$-Fibonacci numbers for~$k \geq 2$.

\begin{theorem} \label{thm:b}
$b_k(n) = |\avoid[k]{n}{122, 213}|$ for all $k \geq 2$ and $n \geq 0$.
\end{theorem}

For example, the $2$-Fibonacci number sequence $\{b_n\}_{n \geq 0}$ is $1, 1, 3, 7, 17, 41, 99,239, 577, \ldots$ (\OEIS{A001333}).
Its first four terms count the $2$-Fibonacci words over $[n]$ for $n = 0,1,2,3$ as shown in~\eqref{eq:b20}--\eqref{eq:b23}.
\begin{align}
1 &= |\avoid[2]{0}{122,213}| = |\{ \epsilon \}| \label{eq:b20} \\
1 &= |\avoid[2]{1}{122,213}| = |\{11\}| \label{eq:b21} \\
3 &= |\avoid[2]{2}{122,213}| = |\{2112, 2121, 2211\}| \label{eq:b22} \\
7 &= |\avoid[2]{3}{122,213}| = |\{322311,332112,323112,332211,323211,332121,323121\}| \label{eq:b23}
\end{align}



Figure \ref{fig:proofab} provides Fibonacci words of both types for $k=3$, and illustrates why Theorems \ref{thm:a} and \ref{thm:b} hold.
Readers who are ready to delve into the proofs of our two main results can safely skip ahead to Section \ref{sec:Fibonaccik}.
The remainder of this section further contextualizes Theorems \ref{thm:a} and \ref{thm:b} and introduces one more result.

\subsection{Classic Pattern Avoidance Results: Fibonacci and Catalan} 
\label{sec:intro_classic}

Theorem \ref{thm:a} provides a $k$-ary generalization of the classic pattern avoidance result by Simion and Schmidt involving permutations and the Fibonacci numbers.
Their statement of the result is provided below.

\begin{theorem}[Proposition 15 in \cite{simion1985restricted}] \label{thm:SS}
For every $n \geq 1$, 
\begin{equation} \label{eq:SS}
    |\avoid{n}{123, 132, 213}| = F_{n+1}
\end{equation}
where $\{F_n\}_{n \geq 0}$ is the Fibonacci sequence, initialized by $F_0 = 0$, $F_1 = 1$.
\end{theorem}

When comparing Theorems \ref{thm:a} and \ref{thm:SS}, note that permutations are $1$-regular words, and they all avoid~$121$.
Thus, $\avoid{n}{123, 132, 213} = \avoid[1]{n}{123, 132, 213} = \avoid[1]{n}{121, 123, 132, 213}$.
In other words, the $121$ pattern in Theorem \ref{thm:a} is hidden in the special case of $k=1$ in Theorem \ref{thm:SS}.
Also note that \eqref{eq:SS} has off-by-one indexing (i.e., subscript $n$ versus $n+1$) and it holds for $n=0$ despite the stated $n \geq 1$ bound.

An even earlier result on pattern avoiding permutations is stated below.

\begin{theorem}[\cite{macmahon1915combinatory1} and \cite{knuth1968art}] \label{thm:MK}
For every $n \geq 0$, 
\begin{equation} \label{eq:MK}
    |\avoid{n}{123}| = \CatNum{n} \text{\qquad and \qquad} |\avoid{n}{213}| = \CatNum{n}
\end{equation}
where $\{\CatNum{n}\}_{n \geq 0}$ is the Catalan sequence starting with $\CatNum{0} = 1$ and $\CatNum{1} = 1$.
\end{theorem}

Permutations avoid $122$, so $\avoid{n}{213} = \avoid[1]{n}{122,213}$.
Thus, the patterns in Theorem \ref{thm:b} are equivalent to those in Theorem \ref{thm:MK} when $k=1$.
But Theorem \ref{thm:MK} is not a special case of Theorem \ref{thm:b}, as our new result only holds when $k \geq 2$.
This gap makes sense as the Catalan numbers do not follow a simple two term recurrence like $a_k(n)$ or $b_k(n)$, and the $1$-Fibonacci words $\avoid[1]{n}{122, 213} = \avoid{n}{213}$ are Catalan words.

\subsection{Base Cases: \texorpdfstring{$(0,1)$}{(0,1)}-Based or \texorpdfstring{$(1,1)$}{(1,1)}-Based} 
\label{sec:intro_base}

Simion and Schmidt's result uses the customary base cases of $F_0 = 0$ and $F_1 = 1$ for Fibonacci~numbers.
However, we use base cases of $a_k(0) = b_k(0) = 1$ and $a_k(1) = b_k(1) = 1$ in our $k$-ary generalizations.

The distinction between \emph{$(0,1)$-based} and \emph{$(1,1)$-based} sequences can be dismissed as cosmetic for the Fibonacci-$k$ recurrence $a_k(n) = a_k(n-1) + k \cdot a_k(n-2)$ since the resulting sequences $0,1,1,k,\ldots$ and $1,1,k,\ldots$ coincide from the first $1$.
The resulting sequences are indeed \emph{shifted} by one index relative to each other, which is important if we want to be sensitive to the off-by-one indexing issue found in~\eqref{eq:SS}.

In contrast, the $n=0$ term is critical to the $k$-Fibonacci recurrence $b_k(n) = k \cdot b_k(n-1) + b_k(n-2)$ since the $(0,1)$-based sequence $0,1,k,k^2+1,\ldots$ and the $(1,1)$-based sequence $1,1,k+1,k^2+k+1,\ldots$ diverge if $k \geq 2$.
In particular, the well-known \emph{Pell sequence} $\{P(n)\}_{n \geq 0}$ (\OEIS{A000129}) follows our $b_2(n)$ recurrence with  $P(0)=0$ and $P(1)=1$, so it is not covered by our Theorem \ref{thm:b} as it is $(0,1)$-based.
\cite{hartung2024regular} found pattern avoidance results for $(1,b)$-based Pell sequences with $b \geq 2$.

With apologies to the Pell numbers, we suggest that $(1,1)$-based sequences are more natural, at least in the context of pattern avoidance.
This is due to the unique word of length $n=0$ that avoids all patterns, namely the empty word $\epsilon$.
So an $n=0$ term of $0$ mistakenly counts a singleton set $\{ \epsilon \}$ as an empty set~$\emptyset$.


\subsection{Four Parameter Generalizations beyond \texorpdfstring{$k$}{k}-Fibonacci and Fibonacci-\texorpdfstring{$k$}{k}}
\label{sec:intro_four}

As mentioned earlier, there are many different notions of \emph{generalized Fibonacci numbers}.
To better discuss similar sequences, it is helpful to define the \emph{$(b_0, b_1)$-based $k_1$-Fibonacci-$k_2$ recurrence} as follows,
\begin{equation}
f(n) = k_1 \cdot f(n-1) + k_2 \cdot f(n-2) \text{ with } f(0) = b_0 \text{ and } f(1) = b_1. \label{eq:generalized} 
\end{equation} 
The resulting sequence $\{f(n)\}_{n \geq 0}$ is the \emph{$(b_0, b_1)$-based $k_1$-Fibonacci-$k_2$ sequence}.
When using these terms, we omit $k_1$ and/or $k_2$ when they are equal to $1$.
Thus, our Fibonacci-$k$ numbers can be described as the $(1,1)$-based Fibonacci-$k$ numbers, or as shifted $(0,1)$-based Fibonacci-$k$ numbers (as per Section \ref{sec:intro_base}).
Likewise, our $k$-Fibonacci numbers can be described as $(1,1)$-based $k$-Fibonacci numbers outside of this paper, and they are not equivalent to $(0,1)$-based $k$-Fibonacci numbers when $k \geq 2$.




Table \ref{tab:not} collects previously studied $(b_0, b_1)$-based $k_1$-Fibonacci-$k_2$ sequences that are not covered by our work. 
For example, the aforementioned Pell sequence is the $(0,1)$-based $2$-Fibonacci sequence using our terminology.
More generally, the \emph{$k$th-Fibonacci sequences}%
\footnote{\setlength\fboxsep{0pt}Somewhat confusingly, the title of this well-cited paper is \emph{On the Fibonacci $k$-numbers} and it contains a section titled \emph{$k$-Fibonacci numbers}, but the term used throughout the paper is \emph{$k$th Fibonacci sequence}.  We use the latter in Table \ref{tab:not}.}
in \cite{falcon2007fibonacci} are the $(0,1)$-based $k$-Fibonacci sequences from the previous paragraph.
When preparing Table \ref{tab:not} we found the summary in \cite{panwar2021note} to be very helpful.
Also note that the same four parameters are used in \cite{gupta2012generalized} with their \emph{generalized Fibonacci sequences} being $(a,b)$-based $p$-Fibonacci-$q$ sequences. 


\begin{table}[ht]
\centering
\small
\begin{tabular}{|@{\;\;}c@{\;\;}c@{\;\;\;}c@{\;\;}r@{\;\;\;}c@{\;\;}c@{\;\;}c@{\;\;}|} \hline
\rowcolor{gray!20}
\multicolumn{1}{|c}{Name} & Bases & $k_1$ & $k_2$ & Numbers & \textsc{Oeis}  & \multicolumn{1}{c@{\;\;}|}{Pattern Avoidance} \\ \hline
Pell & \multirow{2}{*}{$(0,1)$} & \multirow{2}{*}{$2$} & \multirow{2}{*}{$1$}  & \multirow{2}{*}{$0, 1, 2, 5, 12, 29, 70, 169, \ldots$} & \multirow{2}{*}{A000129} & $\avoid{n}{123,2143,3214}$ \\
($2{\text{nd}}$ Fibonacci)                 &         &  &         &                                                       &         & or $\avoid{n}{2\dot{1},231,4321}$ \\
$3{\text{rd}}$ Fibonacci    & $(0,1)$ & $3$ & $1$  & $0, 1, 3, 10, 33, 109, 360, 1189, \ldots$    &\OEIS{A006190} &  \\ 
$4{\text{th}}$ Fibonacci    & $(0,1)$ & $4$ & $1$  & $0, 1, 4, 17, 72, 305, 1292, 5473, \ldots$    &\OEIS{A001076} &  \\ 
$5{\text{th}}$ Fibonacci    & $(0,1)$ & $5$ & $1$  & $0, 1, 5, 26, 135, 701, 3640, 18901, \ldots$    &  &  \\ 
$6{\text{th}}$ Fibonacci    & $(0,1)$ & $6$ & $1$  & $0, 1, 6, 37, 228, 1405, 8658, 53353, \ldots$    &\OEIS{A005668} &  \\ 
$7{\text{th}}$ Fibonacci    & $(0,1)$ & $7$ & $1$  & $0, 1, 7, 50, 357, 2549, 18200, 129949, \ldots$    &  &  \\ 
$8{\text{th}}$ Fibonacci    & $(0,1)$ & $8$ & $1$  & $0, 1, 8, 65, 528, 4289, 34840, 283009, \ldots$    &  &  \\ 
$9{\text{th}}$ Fibonacci    & $(0,1)$ & $9$ & $1$  & $0, 1, 9, 82, 747, 6805, 61992, 564733, \ldots$    &\OEIS{A099371} &  \\ 
Mersenne         & $(0,1)$ & $3$ & ${-}2$ & $0, 1, 3, 7, 15, 31, 63, 127, \ldots$    &\OEIS{A000225} & $\avoid{n}{2\dot{1},132,213}$ \\ 
Lucas            & $(2,1)$ & $1$ & $1$  & $2, 1, 3, 4, 7, 11, 18, 29, \ldots$ &\OEIS{A000032} & $d_n(2413)$ \\ 
Jacobsthal-Lucas & $(2,1)$ & $1$ & $2$  & $2, 1, 5, 7, 17, 31, 65, 127,\ldots$ &\OEIS{A014551} & \\
Pell-Lucas       & $(2,2)$ & $2$ & $1$  & $2, 2, 6, 14, 34, 82, 198, 478, \ldots$ &\OEIS{A002203} & \\ 
\hline
\end{tabular}
\caption{Previously studied $(b_0, b_1)$-based $k_1$-Fibonacci-$k_2$ sequences that are not $a_k(n)$ or $b_k(n)$.
Each row contains a sequence that is not covered by our results, along with its alternate name.
For example, the Mersenne numbers can also be described as the $(0,1)$-based $3$-Fibonacci-$(-2)$ numbers. 
Avoidance results are due to \cite{barcucci2006fibonacci}, \cite{baril2011classical} (dotted), and \cite{cratty2016pattern} (double list $d_n$).
}
\label{tab:not}
\end{table}

\subsection{Pattern Avoidance with Stirling Words}
\label{sec:intro_Stirling}

A \emph{Stirling permutation} is typically defined as a word with two copies of each value in $[n]$ and the property that for each $i \in [n]$, the values between the two copies of $i$ are larger than $i$.
In other words, it is a $2$-regular word over $[n]$ that avoids $212$.
The Stirling permutations with $n=3$ are given in~\eqref{eq:Stirling3}.
\begin{align} \label{eq:Stirling3}
\begin{split}
\avoid[2]{3}{212} = \{
& 112233,
112332,
113322,
122133,
122331,
123321,
133122,
133221, \\
& 221133,
221331,
223311,
233211,
331122,
331221,
332211\}
\end{split}
\end{align}
Famously, \cite{gessel1978stirling} proved that $|\avoid[2]{n}{212}| = (2n-1)!!$.
For example, \eqref{eq:Stirling3} verifies that there are $5!! = 5 \cdot 3 \cdot 1 = 15$ such words for $n=3$.
Generalized \emph{$k$-Stirling words} are the words in $\avoid[k]{n}{212}$ using our notation, and were studied under the name \emph{$r$-multipermutations} in \cite{park1994r}.

\cite{kuba2012enumeration} investigated Wilf-equivalence for $k$-Stirling words that avoid a subset of patterns in $\PERMS{3}$.
They proved that there are five $\mathbb{N}$-Wilf classes%
\footnote{\emph{$\mathbb{N}$-Wilf-equivalence} means Wilf-equivalence for all $k \geq 1$.
The other four classes are $0$, $n+k-1$, $1+k \cdot (n-1)$, and $\binom{n+k-1}{k}$.} 
that avoid three such patterns.
That is, there are five counting functions parameterized by $k$ for $|\avoid[k]{n}{212, \alpha, \beta, \gamma}|$ with distinct $\alpha, \beta, \gamma \in \PERMS{3}$.
In particular, their class $\overline{C}_2$ avoids $\lambda = \{312, 231, 321\}$ and is identical to $\avoid[2]{k}{121,123,132,213}$ used in our Theorem \ref{thm:a}.
This is because $k$-regular word classes respect the symmetries of the square, and complementing our patterns results in $212$ and $\lambda$.
Their proof of this case also uses the same recursive decomposition as our proof of Theorem \ref{thm:a}.
However, its scope and approach obfuscate the simplicity of this important special case. 
For example, the proof of their three-pattern result uses their two-pattern result, so readers must process the special case in stages.
The use of generating functions and partial fraction expansions also belie the simple structure of the strings.
Their representative patterns for $\overline{C}_2$ also requires scaling-up subwords before prefixing $1$s and $2$s (cf., Figure \ref{fig:proofab_Fibk}).
For these reasons, our proof of Theorem \ref{thm:a} is somewhat simpler, even though the same structure is used.

Another result in \cite{kuba2012enumeration} involves the \emph{$k$-Catalan numbers} $\KatNum{k}{n}$, which enumerate $k$-ary trees with $n$ nodes 
and generalize the standard Catalan numbers $\CatNum{n} = \KatNum{2}{n}$ from Section \ref{sec:intro_classic}.
The authors connect these generalized Catalan numbers to Stirling permutations by proving that $|\avoid[k]{n}{212,312}| = \KatNum{k+1}{n}$, with the $k=1$ case providing one generalization of Knuth's contribution to Theorem \ref{thm:MK} (also see Section \ref{sec:intro_regular}).

The aforementioned $k$-Catalan result was also obtained independently within the order theory community.
In particular, the elements of the $m$-Tamari lattice introduced by \cite{bergeron2012higher} are $m$-Dyck paths which are enumerated by $\KatNum{m+1}{n}$. 
As shown by \cite{novelli2020hopf}, the elements can also be viewed as equivalence classes of $(m+1)$-Stirling words (using their $121$-avoiding representation) under a generalization of the sylvester congruence described by \cite{hivert2005algebra};
representatives of these classes are precisely those that avoid the appropriate pattern from $\PERMS{3}$.
We also found \cite{pons2015lattice} and \cite{ceballos2024weak} to be helpful towards understanding these results.

\cite{archer2018pattern} introduced quasi-Stirling permutations: $2$-regular words on $[n]$ that avoid the patterns $1212$ and $2121$. 
The paper shows a bijection between ordered rooted labeled trees and the set of quasi-Stirling permutations, and finds the number of quasi-Stirling permutations that avoid between two and five patterns of length $3$. 
\cite{elizalde2024pattern} consider nonnesting permutations: $2$-regular words on $[n]$ that avoid $1221$ and $2112$, and gives enumeration results for the number of nonnesting permutations that avoid at least two permutations of length $3$.

%


\subsection{Pattern Avoidance with Regular Words}
\label{sec:intro_regular}

There has been a relative lack of pattern avoidance results on non-Stirling regular words.
Said another way, most pattern avoidance results on regular words have (implicitly) assumed that $121$ is avoided.
Omitting this assumption allowed us to discover Theorem \ref{thm:b} --- which uses the non-Stirling pattern $122$ --- as a natural companion to Theorem~\ref{thm:a}.

A significant pattern avoidance result for (non-Stirling) regular words involves the \emph{$k$-ary Catalan numbers} $\KatNum{k}{n}$ which count $k$-ary trees with $n$ nodes.
While Theorem \ref{thm:MK} shows that there are two distinct patterns for $\CatNum{n} = \KatNum{2}{n}$ (up to symmetries of the square), there are three distinct pairs for $k > 2$.
\begin{itemize}
\item \cite{kuba2012enumeration} proved that $|\avoid[k-1]{n}{212,312}| = \KatNum{k}{n}$. 
\item \cite{defant2020stack} proved that $|\avoid[k-1]{n}{221, 231}| = \KatNum{k}{n}$. 
\item \cite{williams2023pattern} proved that $|\avoid[k-1]{n}{112, 123}| = \KatNum{k}{n}$.
\end{itemize}
Note that the first two cases collapse into one case when $k=2$ since complementing and reversing $312$ gives $231$, while the other pattern is hidden. 
Williams also observed that the three results are collectively characterized by choosing a single pattern from $\PERMS{3}$ and a pattern of $1$s and $2$s that is consistent with it\footnote{The author was not initially aware of the Kuba and Panholzer result.
The omission will be corrected in its extended version.}.

We mention that $k$-regular words arise elsewhere under a variety of names, including \emph{uniform permutations}, $m$-permutations, and \emph{fixed-frequency multiset permutations}.
For example, the middle levels theorem by \cite{mutze2016proof,mutze2023book} is conjectured to have a generalization to these words (see \cite{shen2021k}) and significant partial results were proven in a broader context by \cite{gregor2023star}.

\subsubsection{Regular Words and Non-Classical Patterns}
\label{sec:intro_regular_vincular}

To our knowledge, regular words have not been combined with other well-known non-classical pattern avoidance concepts\footnote{For further background on pattern avoidance concepts we recommend \cite{bevan2015permutation}.} (e.g., barred, dotted, mesh, etc).
We argue that such combinations are promising by contributing one more result involving (squared) Fibonacci numbers.

\begin{definition} \label{def:FibSquared}
The \emph{Fibonacci-squared numbers} are $c(n) = a_1(n)^2$.
\end{definition}

In a \emph{fully vincularized pattern} (also known as a \emph{consecutive pattern}) all of the matching symbols must have consecutive positions within the string (see \cite{babson2000generalized}).
In particular, $\vincular{121}$ is present in a string if and only if it contains three consecutive symbols of the form $xyx$ with $y>x$.

\begin{definition} \label{def:FibSquaredWords}
The \emph{Fibonacci-squared words} of length $2n$ are the words in $\avoid[2]{n}{\vincular{121}, 123, 132, 213} \subseteq~\PERMS[2]{n}$.
\end{definition}

Informally, the following theorem states that fully vincularizing the Stirling pattern in Theorem \ref{thm:a} changes the Jacobsthal sequence into the Fibonacci-squared sequence.

\begin{theorem} \label{thm:c}
$c(n) = \avoid[2]{n}{\vincular{121}, 123, 132, 213}$ for all $n \geq 0$.
\end{theorem}

For example, the Fibonacci-squared sequence $\{c_n\}_{n \geq 0}$ is $1, 1, 4, 9, 25, 64, 169, 441, \ldots$ (\OEIS{A007598}).
Its first four terms count the Fibonacci-squared words over $[n]$ for $n = 0,1,2,3$ as shown in~\eqref{eq:c0}--\eqref{eq:c3}.
\begin{align}
1 = |\avoid[2]{0}{\vincular{121}, 123, 132, 213}| &= |\{ \epsilon \}| \label{eq:c0} \\
1 = |\avoid[2]{1}{\vincular{121}, 123, 132, 213}| &= |\{11\}| \label{eq:c1} \\
4 = |\avoid[2]{2}{\vincular{121}, 123, 132, 213}| &= |\{1122, 1221, 2112, 2211\}| \label{eq:c2} \\
9 = |\avoid[2]{3}{\vincular{121}, 123, 132, 213}| &= |\{223311, 233112, 233211, 322311, 323112, 
 \notag \\
&\phantom{{}= |\{} 331122, 331221, 332112, 332211\}| \label{eq:c3}
\end{align}

For example, note that $233112$ appears in \eqref{eq:c3} but not in \eqref{eq:a23}.
This is because it contains the pattern $121$ but not the consecutive pattern $\vincular{121}$.
Figure \ref{fig:proofc} provides Fibonacci-squared words up to $n=5$, and illustrates why Theorem \ref{thm:c} holds.
We prove this result in Section \ref{sec:vincular} using a new recurrence for \OEIS{A007598}.

\subsection{Outline}
\label{sec:intro_outline}

Sections \ref{sec:Fibonaccik} and \ref{sec:kFibonacci} prove Theorems~\ref{thm:a} and \ref{thm:b}, respectively.
The proofs involve simple bijections which could be suitable exercises for students and researchers interested in regular word pattern avoidance. 
Section~\ref{sec:vincular} gives our vincular result on Fibonacci-squared numbers.
Section~\ref{sec:final} closes with final remarks.

As previously mentioned, Theorem \ref{thm:a} was proven in a broader context by \cite{kuba2012enumeration}.
We encourage readers to investigate related results by \cite{janson2011generalized} and \cite{kuba2019stirling}.

\begin{figure}
    \centering
    \begin{subfigure}[T]{0.49\textwidth}
        \centering
        \fbox{\includegraphics{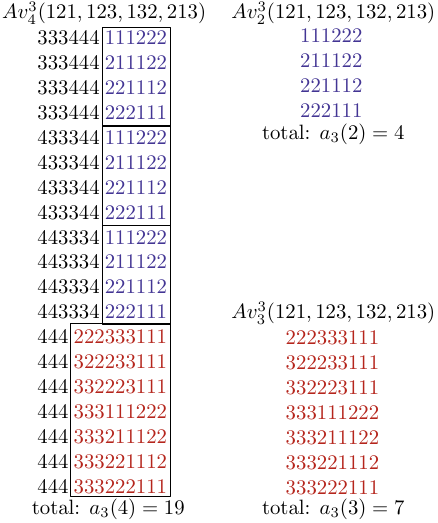}}
        \caption{
        Fibonacci-$k$ words for $k=3$ (in lexicographic order).
        These words contain $k$ copies of $[n]$ and avoid the patterns $121$, $123$, $132$, and $213$.
        Theorem \ref{thm:a} proves that they are enumerated by the $(1,1)$-based Fibonacci-$k$ numbers:
        $a_k(n) = a(n-1) + k \cdot a(n-2)$ with base cases $a_k(0) = a_k(1) = 1$.
        The lists above show how the $n=4$ words are constructed from $k=3$ copies of the $n=2$ words (\blueText{blue}) and one copy of the $n=3$ words (\redText{red}).
        The totals match $a_3(4) = a_3(3) + 3 \cdot a_3(2) = 7 + 3 \cdot 4 = 19$. 
        The special case of $k=1$ is a classic result for Fibonacci numbers: $\avoid{n}{123, 132, 213} = F_{n+1}$ (see Section \ref{sec:intro_classic}).
        }
        \label{fig:proofab_Fibk}
        \vspace{1em}
        \caption{
        $k$-Fibonacci words for $k=3$ (in lexicographic order).
        These words contain $k$ copies of $[n]$ and avoid the patterns $122$ and $213$.
        Theorem \ref{thm:b} proves that they are enumerated by the $(1,1)$-based $k$-Fibonacci numbers when $k \geq 2$:
        $b_k(n) = k \cdot b(n-1) + b(n-2)$ with base cases $b_k(0) = b_k(1) = 1$.
        The lists to the right show how the $n=4$ words are constructed from one copy of the $n=2$ words (\blueText{blue}) and $k=3$ copies of the $n=3$ words (\redText{red}).
        The totals match $b_3(4) = 3 \cdot b_3(3) + b_3(2) = 3 \cdot 13 + 4 = 43$.
        }
        \end{subfigure}
    \hfill
    \begin{subfigure}[T]{0.49\textwidth}
        \centering
        \fbox{\includegraphics{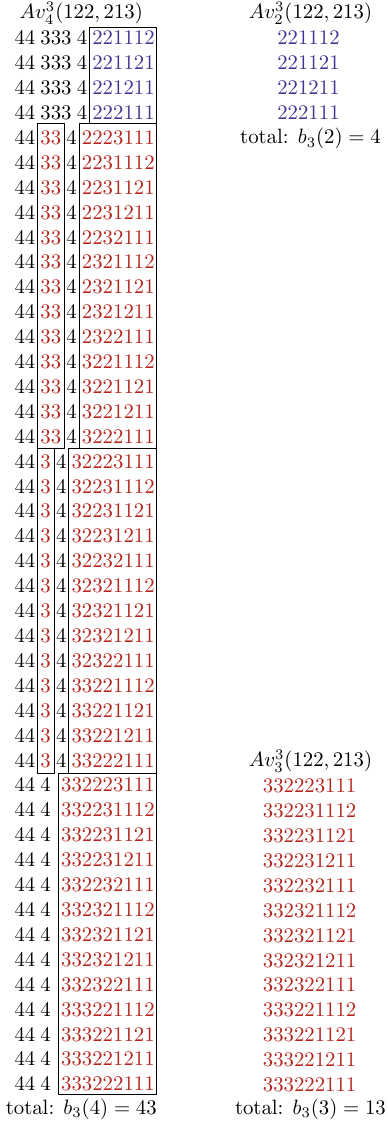}}
        \label{fig:proofab_kFib}
    \end{subfigure}
    \caption{Illustrating Theorems \ref{thm:a} and \ref{thm:b} and their proofs.
    (a) The Fibonacci-$k$ words are $k$-regular words avoiding $\{121, 123, 132, 213\}$ and they are enumerated by the Fibonacci-$k$ numbers $a_k(n)$ for any $k \geq 1$. 
    (b) The $k$-Fibonacci words are $k$-regular words avoiding $\{122, 213\}$ and they are enumerated by the $k$-Fibonacci numbers $b_k(n)$ for any $k \geq 2$. 
    }
    \label{fig:proofab}
\end{figure}

\section{Pattern Avoidance for Fibonacci-\texorpdfstring{$k$}{k} Words}
\label{sec:Fibonaccik}

The Fibonacci-$k$ sequences up to $k=9$ are illustrated in Table \ref{tab:sequencesk}.

\begin{table}[ht]
\centering
\small
\begin{tabular}{|@{\;\;}c@{\;\;}r@{\,}c@{\,}r@{\;\;}c@{\;\;}l@{\;\;}r@{\;\;}|} \hline
\rowcolor{gray!20}
\multicolumn{1}{|c}{Name} & \multicolumn{3}{c}{Recurrence $a(n)$} & Numbers & \multicolumn{1}{c}{\textsc{Oeis}} & \multicolumn{1}{c@{\;\;}|}{Pattern Avoidance} \\ \hline
Fibonacci     & $a(n{-}1)$ & $+$ & $        a(n{-}2)$ & $1, 1, 2,  3,  5,   8,  13,   21, \ldots$ &\OEIS{A000045}$^{\dag}$ & $\avoid{n}{123,132,213}$ \\
Jacobsthal    & $a(n{-}1)$ & $+$ & $2 \cdot a(n{-}2)$ & $1, 1, 3,  5, 11,  21,  43,   85, \ldots$ &\OEIS{A001045}$^{\dag}$ & $\avoid[2]{n}{121,123,132,213}$ \\
Fibonacci-$3$ & $a(n{-}1)$ & $+$ & $3 \cdot a(n{-}2)$ & $1, 1, 4,  7, 19,  40,  97,  217, \ldots$ &\OEIS{A006130} & $\avoid[3]{n}{121,123,132,213}$ \\
Fibonacci-$4$ & $a(n{-}1)$ & $+$ & $4 \cdot a(n{-}2)$ & $1, 1, 5,  9, 29,  65, 181,  441, \ldots$ &\OEIS{A006131} & $\avoid[4]{n}{121,123,132,213}$ \\
Fibonacci-$5$ & $a(n{-}1)$ & $+$ & $5 \cdot a(n{-}2)$ & $1, 1, 6, 11, 41,  96, 301,  781, \ldots$ &\OEIS{A015440} & $\avoid[5]{n}{121,123,132,213}$ \\
Fibonacci-$6$ & $a(n{-}1)$ & $+$ & $6 \cdot a(n{-}2)$ & $1, 1, 7, 13, 55, 133, 463, 1261, \ldots$ &\OEIS{A015441}$^{\dag}$ & $\avoid[6]{n}{121,123,132,213}$ \\
Fibonacci-$7$ & $a(n{-}1)$ & $+$ & $7 \cdot a(n{-}2)$ & $1, 1, 8, 15, 71, 176, 673, 1905, \ldots$ &\OEIS{A015442}$^{\dag}$ & $\avoid[7]{n}{121,123,132,213}$ \\
Fibonacci-$8$ & $a(n{-}1)$ & $+$ & $8 \cdot a(n{-}2)$ & $1, 1, 9, 17, 89, 225, 937, 2737, \ldots$ &\OEIS{A015443} & $\avoid[8]{n}{121,123,132,213}$ \\
Fibonacci-$9$ & $a(n{-}1)$ & $+$ & $9 \cdot a(n{-}2)$ & $1, 1, 10,19, 109,280, 1261,3781, \ldots$ &\OEIS{A015445} & $\avoid[9]{n}{121,123,132,213}$ \\
\hline
\end{tabular}
\caption{The Fibonacci-$k$ numbers $a_k(n) = a_k(n-1) + k \cdot a_k(n-2)$ with $a_k(0) = 1$ and $a_k(1) = 1$ for $1 \leq k \leq 9$ with their sequences and our pattern avoidance results from Theorem \ref{thm:a}. 
Previous pattern-avoiding results are by \cite{simion1985restricted} ($k=1$) (see also \cite{baril2011classical} and \cite{baril2019enumeration}), \cite{mansour2002refined} ($k=2$), and \cite{sun2022d} ($k=3$).
$^{\dag}$\textsc{Oeis} entry starts with~$0$.}
\label{tab:sequencesk}
\end{table}






\begin{lem}\label{lem:order}
    Suppose $\alpha\in \PERMS[k]{n}$ where $\alpha$ avoids the patterns $123$, $132$, $213$ and let $y \in \{2,3,\ldots,n\}$. Then the smallest symbol that can occur before $y$ in $\alpha$ is $y{-}1$. 
\end{lem}

\begin{proof}
    Let $\alpha\in \PERMS[k]{n}$ where $\alpha$ avoids $123$, $132$, $213$ and suppose for contradiction that the symbol $x$ occurs before the symbol $y$ in $\alpha$, where $x<y{-}1$.
    Consider each of the relative positions of the symbol $y{-}1$:
    \begin{itemize} 
        \item $(y{-}1) \cdots x \cdots y$ contains the pattern $213$.
        \item $x \cdots (y{-}1) \cdots y$ contains the pattern $123$.
        \item $x \cdots y \cdots (y{-}1)$ contains the pattern $132$.
    \end{itemize}
    Since each of these patterns must be avoided, we see that $x$ cannot occur before $y$ in $\alpha$.
\end{proof}



\begin{prp} \label{prp:setk} 
Let $n \geq 2$.
Then $\alpha \in \avoid[k]{n}{121, 123, 132, 213}$ if and only if one of the following is true: 
    \begin{enumerate}
        \item[(1)] $\alpha=n^k\gamma$ where $\gamma \in \avoid[k]{n-1}{121, 123, 132, 213}$;
        \item[(2)] $\alpha=n^b(n{-}1)^kn^{k-b}\gamma$ where $0\leq b < k$ and $\gamma \in \avoid[k]{n-2}{121, 123, 132, 213}$. 
    \end{enumerate}
\end{prp}

\begin{proof}
    We start by proving that any element of $\avoid[k]{n}{121, 123, 132, 213}$ is of form (1) or (2) above. Suppose $\alpha=a_1a_2 \cdots a_m \in \avoid[k]{n}{121, 123, 132, 213}$. If the first $k$ symbols are not all $n$, let $a_i$ be the first symbol of $\alpha$ that is not $n$, where $0\leq i<k$. Since at least one copy of $n$ must follow $a_i$, $a_i$ can only be equal to $n{-}1$ by Lemma \ref{lem:order}. We claim that all $k-1$ remaining copies of $n{-}1$ must follow immediately. Suppose for contradiction that a different symbol, $t$, appears before one of the copies of $n{-}1$. Then $t$ cannot be equal to $n$, because more copies of $n{-}1$ must follow, creating the sub-permutation $(n{-}1)n(n{-}1)$, which is the forbidden pattern $121$. Thus all copies of $n{-}1$ must precede the remaining copies of $n$. The symbol $t$ cannot be less than $n{-}1$, because there are remaining copies of $n$ to the right, and by Lemma \ref{lem:order}, the smallest symbol that can occur before $n$ is $n{-}1$. Therefore, if $\alpha$ does not begin with $n^k$, then $\alpha$ must begin with $b$ copies of $n$, followed by $k$ copies of $n{-}1$, followed by $k-b$ copies of $n$, where $0\leq b< k-1$.

    We now show that any element described by (1) or (2) is a member of $\avoid[k]{n}{121, 123, 132, 213}$.
    Observe first that $\alpha=n^k\gamma\in \avoid[k]{n}{121, 123, 132, 213}$ if $\gamma \in \avoid[k]{n-1}{121, 123, 132, 213}$, since none of the patterns we wish to avoid start with a largest symbol.  Suppose that $\alpha$ is a word of the form $n^b(n{-}1)^kn^{k-b}\gamma$, where $\gamma \in \avoid[k]{n-2}{121, 123, 132, 213}$ and  $0\leq b < k$. Observe that $n^b(n{-}1)^kn^{k-b}$ does not contain the pattern $121$, and given that $\gamma$ also did not contain this pattern, $\alpha$ also cannot, since the symbols in $\gamma$ are all smaller than $n$ and $n{-}1$. Similarly, the patterns $123$, $132$, and $213$ also cannot occur in $\alpha$, since the prefix contains the largest values $n$ and $n{-}1$, and the forbidden patterns did not occur in $\gamma$.
\end{proof}  


The following remark and corollary complete our proof of Theorem \ref{thm:a} for Fibonacci-$k$ words.

\begin{rem}\label{base} The set $\avoid[k]{0}{121, 123, 132, 213}=\{ \epsilon\}$; the set $\avoid[k]{1}{121, 123, 132, 213}=\{1^k\}$. 

\end{rem}

\begin{cor} \label{cor:Fibk}
    The following holds for all $k\geq 1$ and $n\geq 2$ where $|\avoid[k]{0}{121, 123, 132, 213}|=1$ and $|\avoid[k]{1}{121, 123, 132, 213}|=1$, 
    \begin{equation}
    |\avoid[k]{n}{121, 123, 132, 213}| = |\avoid[k]{n-1}{121, 123, 132, 213}| + k \cdot |\avoid[k]{n-2}{121, 123, 132, 213}|.
    \end{equation}
\end{cor}

\begin{proof}
By Proposition \ref{prp:setk}, the elements of $\avoid[k]{n}{121, 123, 132, 213}$ can be constructed as follows:
\begin{itemize}
    \item Create $|\avoid[k]{n-1}{121, 123, 132, 213}|$ elements of $\avoid[k]{n}{121, 123, 132, 213}$ by inserting $n^k$ in front of each element of $\avoid[k]{n-1}{121, 123, 132, 213}$;
    \item Create $k\cdot |\avoid[k]{n-2}{121, 123, 132, 213}|$ elements of $\avoid[k]{n}{121, 123, 132, 213}$ by inserting a prefix of the form $n^b(n-1)^kn^{k-b}$ in front of each element of $\avoid[k]{n-2}{121, 123, 132, 213}$ where $0 \leq b < k$. There are $k$ such prefixes. 
\end{itemize}
Therefore, with Remark \ref{base} as a base case, we have our result.     Note that in the second step of constructing $\avoid[k]{2}{121, 123, 132, 213}=\{2^b1^k2^{k-b}:0\leq b \leq k \}$, we create $k$ elements by placing $2^b1^k2^{k-b}$ in front of the empty word, to obtain $2^b1^k2^{k-b}\cdot \epsilon=2^b1^k2^{k-b}$, for each $0 \leq b < k$.
\end{proof}


\section{Pattern Avoidance for \texorpdfstring{$k$}{k}-Fibonacci Words}
\label{sec:kFibonacci}

The $k$-Fibonacci sequences up to $k=9$ are illustrated in Table \ref{tab:ksequences}.

\begin{table}[ht]
\centering
\small
\begin{tabular}{|@{\;\;}c@{\;\;}r@{\,}c@{\,}r@{\;\;}c@{\;\;}c@{\;\;}r@{\;\;}|} \hline
\rowcolor{gray!20}
\multicolumn{1}{|c}{Name} & \multicolumn{3}{c}{Recurrence $b(n)$} & Numbers & \textsc{Oeis}  & \multicolumn{1}{c@{\;\;}|}{Pattern Avoidance} \\ \hline
Catalan       &                    &     &            & $1, 1, 2, 5, 14, 42, 132, 429, \ldots$ &\OEIS{A000108} & $\avoid{n}{132}$ \\ 
$2$-Fibonacci & $2 \cdot b(n{-}1)$ & $+$ & $b(n{-}2)$ & $1, 1, 3, 7, 17, 41, 99, 239, \ldots$ &\OEIS{A001333} & $\avoid[2]{n}{112,132}$ \\
$3$-Fibonacci & $3 \cdot b(n{-}1)$ & $+$ & $b(n{-}2)$ & $1, 1, 4, 13, 43, 142, 469, 1549, \ldots$ &\OEIS{A003688} & $\avoid[3]{n}{112,132}$ \\
$4$-Fibonacci & $4 \cdot b(n{-}1)$ & $+$ & $b(n{-}2)$ & $1, 1, 5, 21, 89, 377, 1597, 6765, \ldots$ &\OEIS{A015448} & $\avoid[4]{n}{112,132}$ \\
$5$-Fibonacci & $5 \cdot b(n{-}1)$ & $+$ & $b(n{-}2)$ & $1, 1, 6, 31, 161, 836, 4341, 22541, \ldots$ &\OEIS{A015449} & $\avoid[5]{n}{112,132}$ \\
$6$-Fibonacci & $6 \cdot b(n{-}1)$ & $+$ & $b(n{-}2)$ & $1, 1, 7, 43, 265, 1633, 10063, 62011, \ldots$ &\OEIS{A015451} & $\avoid[6]{n}{112,132}$ \\
$7$-Fibonacci & $7 \cdot b(n{-}1)$ & $+$ & $b(n{-}2)$ & $1, 1, 8, 57, 407, 2906, 20749, 148149, \ldots$ &\OEIS{A015453} & $\avoid[7]{n}{112,132}$ \\
$8$-Fibonacci & $8 \cdot b(n{-}1)$ & $+$ & $b(n{-}2)$ & $1, 1, 9, 73, 593, 4817, 39129, 317849, \ldots$ &\OEIS{A015454} & $\avoid[8]{n}{112,132}$ \\
$9$-Fibonacci & $9 \cdot b(n{-}1)$ & $+$ & $b(n{-}2)$ & $1, 1, 10, 91, 829, 7552, 68797, 626725, \ldots$ &\OEIS{A015455} & $\avoid[9]{n}{112,132}$ \\
\hline
\end{tabular}
\caption{The $k$-Fibonacci numbers $b_k(n) = k \cdot b_k(n-1) + b_k(n-2)$ with $b_k(0) = 1$ and $b_k(1) = 1$ for $2 \leq k \leq 9$ with their sequences and our pattern avoidance results from Theorem \ref{thm:b}. 
Previous pattern-avoiding results for $k=2$ are by \cite{gao2011sequences}, \cite{kuba2012enumeration}, and \cite{gao2014tackling}. 
The Catalan numbers do not follow a Fibonacci-style recurrence, and are not covered by Theorem \ref{thm:b}. 
However, they are included in the table as they count $1$-regular words avoiding $112$ and $132$ (with the former hidden).
Many pattern-avoiding results exist for the Catalan numbers, including \cite{knuth1968art}, \cite{macmahon1915combinatory1}, and \cite{baril2011classical}.}
\label{tab:ksequences}
\end{table}





\begin{lem}\label{consecutive2}
    Suppose $\beta \in \avoid[k]{n}{122, 213}$ with $n \geq 2$ and $k \geq 2$. 
    For any $x,y \in \{1,2,\ldots,n\}$ with $y>x$, at most one copy of $y$ can occur to the right of $x$ in $\beta$.
\end{lem}

\begin{proof}
    Suppose to the contrary that $\beta \in \avoid[k]{n}{122, 213}$, and in $\beta$, some symbol $x<y$ occurs with more than one copy of $y$ to the right of $x$. 
    Then $\beta$ contains the pattern $122$, which contradicts that $\beta \in \avoid[k]{n}{122, 213}$.
\end{proof}

\begin{lem}\label{order2}
    Suppose $\beta \in \avoid[k]{n}{122, 213}$ with $n \geq 2$ and $k \geq 2$ and $y \in \{2,3,\ldots,n\}$. 
    Then the smallest symbol that can occur before $y$ in $\beta$ is $y{-}1$.
\end{lem}

\begin{proof}
    Let $\beta \in \avoid[k]{n}{122, 213}$ and suppose that $x$ occurs before $y$ in $\beta$, where $x<y{-}1$. 
    By Lemma \ref{consecutive2}, at most one copy of $y{-}1$ occurs to the right of $x$, and thus at least $k-1$ copies of $y{-}1$ occur to the left of $x$. 
    But then $\beta$ contains $(y{-}1) \cdots x \cdots y$, which is $213$.
\end{proof}


Let $\alpha=a_1a_2 \cdots a_n$ be a word. We define $\alpha'=$ insert$(\alpha,x, i)$ to be the result of inserting the symbol $x$ into position $i$ of $\alpha$. That is, insert$(\alpha,x,i)=a_1a_2 \cdots a_{i-1}xa_{i} \cdots a_n$.  


\begin{prp}\label{prp:set}
    Let $n \geq 2$ and $k \geq 2$. Then $\beta \in \avoid[k]{n}{122, 213}$ if and only if one of the following is true:
    \begin{enumerate}
        \item[(1)] $\beta = n^{k-1} \alpha'$ for some $\alpha'=$ insert$(\alpha,n,i+1)$ with $0 \leq i \leq k-1$ and $\alpha \in \avoid[k]{n-1}{122, 213}$;
        \item[(2)] $\beta=n^{k-1}(n{-}1)^kn\gamma$ where $\gamma \in \avoid[k]{n-2}{122, 213}$.
    \end{enumerate}
\end{prp}

\begin{proof}
    We start by proving that any element of $\avoid[k]{n}{122, 213}$ is of form (1) or (2) above. Let $\beta \in \avoid[k]{n}{122, 213}$ where $k\geq 2$. By Lemma \ref{consecutive2}, for any $x<n$, at most one copy of $n$ occurs to the right of $x$, and thus at least $k-1$ copies of $n$ occur to the left of $x$. Therefore, $\beta$ must begin with $n^{k-1}$. Consider the position of the $k$th copy of $n$. The smallest symbol that can be to its left is $n{-}1$ by Lemma \ref{order2}. If all $k$ copies of $n{-}1$ occur before the $k$th copy of $n$, then $\beta=n^{k-1}(n{-}1)^kn\gamma$ where $\gamma \in \avoid[k]{n-2}{122, 213}$. 
    Otherwise, suppose $i<k$ copies of $n{-}1$ occur before the $k$th copy of $n$. Then $\beta$ must begin with $n^{k-1}(n{-}1)^in(n{-}1)^{k-i-1}$, since at least $k-1$ copies of $n{-}1$ must precede all smaller symbols by Lemma \ref{consecutive2}. Note that any element $\alpha$ of $\avoid[k]{n-1}{122, 213}$ begins with $(n{-}1)^{k-1}$, again by Lemma \ref{consecutive2}. Therefore, $\beta$ is of the form $n^{k-1} \alpha'$ for some $\alpha'=$ insert$(\alpha,n,i+1)$ with $0 \leq i \leq k-1$ and $\alpha \in \avoid[k]{n-1}{122, 213}$.


    We now show that any element described by (1) or (2) is a member of $\avoid[k]{n}{122, 213}$. Consider a word of the form $n^{k-1}(n{-}1)^kn\gamma$ where $\gamma \in \avoid[k]{n-2}{122, 213}$. 
    The prefix $n^{k-1}(n{-}1)^kn$ does not contain $122$ or $213$, and since only smaller symbols occur in $\gamma$, $n^{k-1}(n{-}1)^kn\gamma$ also avoids these patterns. Therefore, $n^{k-1}(n{-}1)^kn\gamma \in \avoid[k]{n}{122, 213}$. Similarly, suppose $\alpha$ is in $\avoid[k]{n-1}{122, 213}$, and $\alpha'=$ insert$(\alpha,n,i+1)$, where $0 \leq i \leq k-1$. Consider $n^{k-1}\alpha'$. Since $\alpha \in \avoid[k]{n-1}{122, 213}$, $\alpha$ avoids the patterns $122$ and $213$. Inserting a copy of $n$ within $(n{-}1)^{k-1}$ to create $\alpha'$ will not create either of these patterns, nor will inserting $n^{k-1}$ in front of $\alpha'$. Thus, $n^{k-1}\alpha'\in \avoid[k]{n}{122, 213}$. 
\end{proof}

The following remark and corollary complete our proof of Theorem \ref{thm:b} for $k$-Fibonacci words.

\begin{rem}\label{base2}
    The set $\avoid[k]{0}{122, 213}=\{\epsilon\}$, where $\epsilon$ is the empty word. The set $\avoid[k]{1}{122, 213}=\{1^k\}$. 
\end{rem}

\begin{cor} \label{cor:kFib}
    This holds for $n \geq 2$ and $k \geq 2$ where $|\avoid[k]{0}{122, 213}|=1$ and $|\avoid[k]{1}{122, 213}|=~1$,
    \begin{equation}
    |\avoid[k]{n}{122, 213}| = k \cdot|\avoid[k]{n-1}{122, 213}| +  |\avoid[k]{n-2}{122, 213}|.
    \end{equation}
\end{cor}

\begin{proof}
By Proposition \ref{prp:set}, the elements of $\avoid[k]{n}{122, 213}$ can be constructed as follows:
\begin{itemize}
    \item Create $k \cdot|\avoid[k]{n-1}{122, 213}|$ elements of $\avoid[k]{n}{122, 213}$ by placing $n^{k-1}$ in front of each element of $\alpha'=$ insert$(\alpha,n,i+1)$, for each $\alpha \in \avoid[k]{n-1}{122, 213}$ and each $0 \leq i \leq k-1$;
    \item Create $|\avoid[k]{n-2}{122, 213}|$ elements of $\avoid[k]{n}{122, 213}$ by placing a prefix of the form $n^{k-1}(n{-}1)^kn$ in front of each element of $\avoid[k]{n-2}{122, 213}$. 
\end{itemize}

Therefore, with Remark \ref{base2} as a base case, we have our result. 
Note that in the second step of constructing $\avoid[k]{2}{122, 213}=\{2^{k-1}1^b21^{k-b}: 0 \leq b \leq k\}$, we create $|\avoid[k]{0}{122, 213}|=1$ element by placing $2^{k-1}1^k2$ in front of the empty word, to obtain $2^{k-1}1^k2\cdot \epsilon=2^{k-1}1^k2$.
\end{proof}

\section{Pattern Avoidance for the Fibonacci-Squared Sequence}
\label{sec:vincular}

One motivation of this paper is to promote the use of $k$-regular words in the pattern avoidance community.
To further this goal, we demonstrate here that $k$-regular words can be combined with non-classical patterns to produce interesting enumeration results.

A \emph{vincular pattern} allows pairs of adjacent symbols in the pattern to be specified as being \emph{attached}, meaning that a subword only contains the pattern if its corresponding pair of symbols are adjacent.
Theorem \ref{thm:c} (see Section \ref{sec:intro_regular_vincular}) modifies Theorem \ref{thm:a} by forcing the $121$ pattern in $\avoid[2]{n}{121,123,132,213}$ to be fully attached or consecutive (i.e., both $12$ and $21$ are attached), and states that the resulting Fibonacci-squared words are counted by the Fibonacci-squared sequence $c(n) = a_1(n)^2$.


The Fibonacci-squared sequence is known to satisfy a number of different formulae, including the following three term recurrence (see \OEIS{A007598}).
\begin{equation} \label{eq:cprime}
c'(n) = 2 \cdot c'(n{-}1) + 2 \cdot c'(n{-}2) - c'(n{-}3) \text{ with } c'(0) = c'(1) = 1 \text{ and } c'(2)= 4.
\end{equation}
We begin this section by proving that the Fibonacci-squared sequence also obeys the following recurrence.
\begin{equation} \label{eq:cprimeprime}
c''(n) = c''(n{-}1) + 3 \cdot c''(n{-}2) + 2\cdot \sum_{i=3}^{n} c''(n-i) \text{ with } c''(0) = c''(1) = 1.
\end{equation}

\begin{lem} \label{lem:recurrences}
The Fibonacci-squared numbers follow the recurrences in \eqref{eq:cprime} and \eqref{eq:cprimeprime}.  
That is, $c(n) = c'(n) = c''(n)$ for all $n \geq 0$. 
\end{lem}

\begin{proof}
As previously mentioned, $c(n) = c'(n)$ is known (see \OEIS{A007598}) and we prove the second equality.
First observe that $c'(0)=1=c''(0)$ and $c'(1)=1=c''(1)$, and that $c''(2)=c''(1)+3c''(0)=4$, which is equal to $c'(2)$.
Suppose that $c'(k)=c''(k)$ for all $k$ with $3 \leq k \leq n-1$ and consider $c'(n)=2\cdot c'(n-1)+2\cdot c'(n-2)-c'(n-3)$. This is equal to 
    
$c'(n-1)+2\cdot c'(n-2)-c'(n-3)+c'(n-1)$, which by the inductive hypothesis, 
\begin{align*}
&= c''(n-1)+2\cdot c''(n-2)-c''(n-3)+c''(n-1)\\
&= c''(n-1)+2\cdot c''(n-2)-c''(n-3)+c''(n-2)+3 \cdot c''(n-3)+2\cdot \sum_{i=3}^{n-1} c''(n-1-i)\\
&= c''(n-1)+3 \cdot c''(n-2)+2 \cdot c''(n-3)+2\cdot \sum_{i=3}^{n-1} c''(n-1-i)\\
&= c''(n-1)+3 \cdot c''(n-2)+2\cdot \sum_{i=3}^{n} c''(n-i)=c''(n). \qed
\end{align*}

Ultimately, we will prove that Fibonacci-squared words can be decomposed into smaller words according to \eqref{eq:cprimeprime}.
Towards this decomposition, we make several definitions involving regular words below.

Given a non-empty $k$-regular word over $[n]$, its \emph{base} is its longest suffix that is a $k$-regular word over~$[m]$ for some $m < n$.
That is, if $\gamma \in \PERMS[k]{n}$ is non-empty, then its base is the word $\beta \in \PERMS[k]{m}$ that maximizes $m$ subject to $\gamma = \alpha \beta$ and $0 \leq m < n$.
The assumption that $\gamma \in \PERMS[k]{n}$ is non-empty is equivalent to $n > 0$ (with the implicit assumption that $k \geq 1$).
The base is well-defined for all non-empty $k$-regular words.
This is because all such $\gamma$ have a shortest suitable suffix, namely $\beta = \epsilon$.
Our definition also ensures that the base is a \emph{strict suffix} (i.e., it is not equal to entire word).
As a result, the remaining prefix is non-empty, and we refer to it as the \emph{annex}.
Equivalently, the annex is the shortest non-empty prefix consisting of $k$ copies of each symbol in $\{n, n-1, \ldots, m+1\}$ for some $m < n$.
Given a non-empty regular word $\gamma \in \PERMS[k]{n}$, its \emph{standard partition} is $\gamma = \alpha \beta$, where $\beta$ is its base and $\alpha$ is its annex.

Standard partitions are illustrated for several regular words $\gamma$ below.
When considering these examples, remember that the base $\beta$ is chosen to be as long as possible without equalling $\gamma$ and that $\beta = \epsilon$ is allowed.

\begin{align*}
\gamma = \underbracket{\,6\,5\,6\,5\,5\,6\,}_{\mbox{\normalsize$\alpha$}}\underbracket{\,4\,4\,3\,2\,3\,4\,3\,2\,2\,1\,1\,1\,}_{\mbox{\normalsize$\beta$}} \in \PERMS[3]{6} 
&&
\gamma = \underbracket{\,1\,2\,3\,4\,5\,6\,}_{\mbox{\normalsize$\alpha$}}\underbracket{\phantom{\epsilon\epsilon\epsilon}}_{\mbox{\normalsize$\beta$}} \in \PERMS[1]{6} 
&&
\gamma = \underbracket{\,3\,3\,3\,}_{\mbox{\normalsize$\alpha$}}\underbracket{\,2\,2\,2\,1\,1\,1\,}_{\mbox{\normalsize$\beta$}} \in \PERMS[3]{3}
\end{align*}

Now we return our attention to Fibonacci-squared words, which are a subset of the $2$-regular words.
Since a Fibonacci-squared word avoids $\overline{121}$, $123$, $132$, $213$, so too does every prefix and suffix of it.
In particular, the both the annex and the base of any Fibonacci-squared word avoid all four patterns.
Due to the structure of the patterns, the converse is also true.
\end{proof}

\begin{lem} \label{lem:composition}
If a non-empty $2$-regular word $\gamma \in \PERMS[2]{n}$ has standard partition $\gamma = \alpha \beta$ and both $\alpha$ and $\beta$ avoid the patterns $\vincular{121}$, $123$, $132$, $213$, then $\gamma$ is a Fibonacci-squared word.
That is, $\gamma \in \avoid[2]{n}{\vincular{121},123,132,213}$.
\end{lem}

\begin{proof}
Each of the four patterns has a smaller symbol followed by a larger symbol. 
If $\alpha$ and $\beta$ each avoid these patterns, then it is impossible for $\alpha$ followed by $\beta$ to contain any of these patterns, since the symbols in $\alpha$ are all larger than the symbols in $\beta$. 
\end{proof}

Now we make final preparations before proving the main result of this section.
Our proof considers an arbitrary Fibonacci-squared word $\gamma \in \avoid[2]{n}{\vincular{121},123,132,213}$ from left-to-right.
We argue that its annex is either one of four special cases, or can be described as a path within a particular tree that we now define.
The \emph{prefix-tree} $\prefixTree{n}$ is a bi-rooted tree 
with $4(n-1)$ vertices 
with labels taken from $[n]$.
It contains a primary path that follows a \emph{two steps forward one step back} pattern from $n$ down to $1$.
\begin{equation} \label{eq:2forward1backward}
n, n-2, n-1, n-3, n-2, n-4, n-3, \ldots, 5, 6, 4, 5, 3, 4, 2, 3, 1.
\end{equation}


Every second node $i$ on the primary path (i.e., $n-2, n-3, n-4, \ldots, 2, 1$) has an edge to another copy of $i$ which has an edge to another copy of $i+1$.
Finally, one root is labeled $n$ and has an edge to another copy of $n-1$, while a second root is labeled $n-1$ and has an edge to another copy of $n$, and both of these non-root nodes have an edge to the start of the primary path.
In the prefix tree, the label 1 appears twice, the labels 2 and $n$ appear three times, and all other labels appear 4 times. Thus the tree contains $2+2\cdot3+4(n-3)=8+4n-12=4(n-1)$ vertices.
Figure \ref{fig:proofc} illustrates $\prefixTree{n}$ and $\prefixTree{4}$.


\begin{theorem} \label{thm:cprimeprime}
    Any word $ \gamma \in \avoid[2]{n}{\vincular{121},123,132,213}$ has the standard partition $\alpha \beta$ where $\alpha$ is either a root-to-leaf path on the prefix tree, or is one of the four special cases $nn, (n{-}1)(n{-}1)nn$, $(n{-}1)nn(n{-}1)$, $n(n{-}1)(n{-}1)n$.

\end{theorem}

\begin{proof} We describe all possible annexes for a word $ \gamma \in \avoid[2]{n}{\vincular{121},123,132,213}$. 
Observe first that $nn$ is an annex. Next suppose $\gamma \in \avoid[2]{n}{\vincular{121},123,132,213}$ and $\gamma$ does not begin with $nn$.
By Lemma \ref{lem:order}, the only element that can precede $n$ is $n{-}1$. 
We first consider length-4 prefixes where all symbols are from $\{n{-}1, n\}$. 
There are $\frac{4!}{2!2!}=6$ possible such prefixes; 5 that do not start with $nn$. 
Of the remaining possibilities, both $(n{-}1)n(n{-}1)n$ and $n(n{-}1)n(n{-}1)$ contain $\vincular{121}$, and thus are excluded.
This leaves three possibilities, $(n{-}1)(n{-}1)nn$, $(n{-}1)nn(n{-}1)$, and $n(n{-}1)(n{-}1)n$, which are all annexes. 

The annexes described so far are the four special cases not generated by the prefix tree. Next suppose $\gamma \in \avoid[2]{n}{\vincular{121},123,132,213}$ and the first four symbols are not all from $\{n{-}1, n\}$. Both copies of $n$ must occur in the first four positions since the only symbol that can precede $n$ is $n{-}1$. The only symbol that can precede $n{-}1$ is $n{-}2$. Therefore, if $\gamma$ is not in one of the special cases, the first four symbols of $\gamma$ consist of two copies of $n$, one copy of $n{-}1$, and one copy of $n{-}2$. 
Since $n{-}2$ cannot precede $n$, $n{-}2$ cannot occur in the first two positions. Further, $n{-}2$ in the third position results in  $nn(n{-}2)(n{-}1)$, the already considered annex $nn$. Thus, $n{-}2$ can only be in the fourth position. There are two ways to place $n$, $n$, $n-1$ in the first three positions that avoid the annex $nn$. 
Thus, if $\gamma$ does not begin with one of the four special cases, then $\gamma$ begins with $(n{-}1)nn(n{-}2)$ or $n(n{-}1)n(n{-}2)$. Observe that the only difference between these two strings is whether the first two symbols are $(n{-}1)n$ or $n(n{-}1)$, the two roots of the prefix tree.


Continuing from $(n{-}1)n$ or $n(n{-}1)$, the remaining copy of $n{-}1$ can only be preceded by $n{-}2$. 
Thus, there are only two ways for $\gamma$ to continue: $(n{-}2)(n{-}1)$ or $(n{-}1)$. The first case results in the annexes  $(n{-}1)nn(n{-}2)(n{-}2)(n{-}1)$ and $n(n{-}1)n(n{-}2)(n{-}2)(n{-}1)$, corresponding with the first leaf on the prefix tree.
In the second case, $(n{-}1)nn(n{-}2)$ or $n(n{-}1)n(n{-}2)$ continues with $(n{-}1)$. 
The symbol $n{-}2$ is not yet paired, and another $n{-}2$ cannot follow immediately, because this would create the forbidden $\vincular{121}$.  All copies of symbols greater than $n{-}2$ have already occurred, and the smallest symbol that can proceed $n{-}2$ is $n{-}3$, so $n{-}3$ must be next. It can either be that both copies of $n{-}3$ precede the second copy of $n{-}2$, or a single copy precedes $n{-}2$. 
The first case creates the annexes $(n{-}1)nn(n{-}2)(n{-}1)(n{-}3)(n{-}3)(n{-}2)$ and $n(n{-}1)n(n{-}2)(n{-}1)(n{-}3)(n{-}3)(n{-}2)$, corresponding with the second leaf on the prefix tree. 
Otherwise, we continue down the primary path with a single copy of $n{-}3$ preceding $n{-}2$.

We are now in the same position as in the previous step. The $n{-}3$ is unpaired, but placing $n{-}3$ immediately would create $\vincular{121}$. We can either continue to the next leaf with $(n{-}4)(n{-}4)(n{-}3)$, creating two annexes, or continue down the primary path with $(n{-}4)(n{-}3)$, in which case we are in an equivalent position.

At each step, root $(n{-}1)nn \cdots $ or $n(n{-}1)n \cdots $ ends with $(n{-}k+1)(n-{k}+2)$, and can either continue with the leaf $(n{-}k)(n{-}k)(n{-}k+1)$, or down the primary path with $(n{-}k)(n{-}k+1)$. The first case creates two annexes. The second case, continues down the primary path and encounters an equivalent pair of choices.  This pattern continues down to $k=n{-}1$, where the annex is now a word in $\avoid[2]{n}{\vincular{121},123,132,213}$, and the base is the empty word. We have found all possible annexes for $\gamma$. 
\end{proof}

Corollary \ref{cor:cprimeprime} restates Theorem \ref{thm:c} using Lemma \ref{lem:recurrences}.

\begin{cor} \label{cor:cprimeprime}
    $c''(n) = |\avoid[2]{n}{\vincular{121},123,132,213}|$ for all $n \geq 0$.
\end{cor}

\begin{proof} We proceed by strong induction on $n$.
The base cases of $n=0$ and $n=1$ hold by the following.
\begin{itemize}
    \item $|\avoid[2]{0}{\vincular{121},123,132,213}| = |\{ \epsilon \}| = 1 = c''(0)$.
    \item $|\avoid[2]{0}{\vincular{121},123,132,213}| = |\{ 11 \}| = 1 = c''(1)$.
\end{itemize}

Now we must prove the following.
\begin{equation} \label{eq:cGoal}
|\avoid[2]{n}{\vincular{121},123,132,213}| = c''(n) = c''(n-1) + 3 \cdot c''(n-2) + 2 \sum_{i=3}^{n} c''(n-i).
\end{equation}

Consider each of the possible annexes for $\gamma \in \avoid[2]{n}{\vincular{121},123,132,213}|$. 
Words of the form $nn\beta$ where $\beta \in \avoid[2]{n{-}1}{\vincular{121},123,132,213}$ give $|\avoid[2]{n{-}1}{\vincular{121},123,132,213}|$ elements to $\avoid[2]{n}{\vincular{121},123,132,213}$. This accounts for the $c''(n-1)$ term in \eqref{eq:cGoal}.
The three annexes $(n{-}1)(n{-}1)nn$, $(n{-}1)nn(n{-}1)$, and $n(n{-}1)(n{-}1)n$ each contribute $|\avoid[2]{n{-}2}{\vincular{121},123,132,213}|$ elements to $\avoid[2]{n}{\vincular{121},123,132,213}$. This accounts for the $3c''(n-2)$ term in \eqref{eq:cGoal}.

For each $2 \leq k \leq n-1$, there are two root-to-leaf paths in the prefix tree. These are the annexes $(n{-}1)nn \cdots (n{-}k)(n{-}k)(n{-}k+1)$ and $n(n{-}1)n \cdots (n{-}k)(n{-}k)(n{-}k+1)$, each of which composes with elements of $\avoid[2]{n-k-1}{\vincular{121},123,132,213}$. These are the root-to-leaf paths in the prefix tree, and together, these contribute $\displaystyle 2 \cdot \sum_{i=3}^{n} c''(n-i)$ to $\eqref{eq:cGoal}$.
\end{proof}

\begin{figure}
    \centering
    \begin{subfigure}[T]{0.505\textwidth}
        \centering
        \fbox{\includegraphics[scale=0.975]{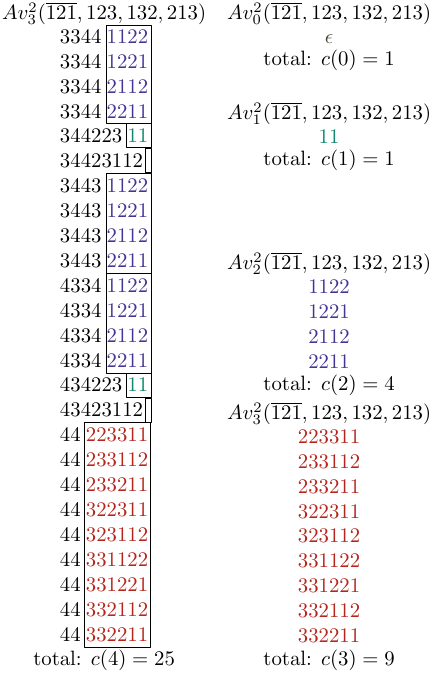}}
        \caption{
        Fibonacci-squared words (in lexicographic order).
        These words contain $2$ copies of $[n]$ and avoid the consecutive pattern $\vincular{121}$ and the classic patterns $123$, $132$, and $213$.
        Theorem \ref{thm:c} proves that they are enumerated by the Fibonacci-squared numbers $c(n) = a_1(n)^2$, which satisfy the recurrence $c''(n) = c''(n-1) + 3 \cdot c''(n-2) + 2\sum_{i=3}^{n}{c''(n-i)}$ as proven in Lemma \ref{lem:recurrences}.
        The lists above show how the $n=4$ words are constructed from two copies of the $n-4=0$ words (grey), two copies of the $n-3=1$ words (\greenText{green}), three copies of the $n-2=2$ words (\blueText{blue}), and one copy of the $n-1=3$ words (\redText{red}).
        The totals match $c(4) = c(3) + 3 \cdot c(2) + 2 \cdot c(1) + 2 \cdot c(0) = 9 + 3 \cdot 4 + 2 \cdot 1 + 2 \cdot 1 = 25$.
        }
        \label{fig:proofc_FibSquared}
    \end{subfigure}
    \hfill
    \begin{subfigure}[T]{0.485\textwidth}
        \centering
        \fbox{\includegraphics[align=t, scale=0.78]{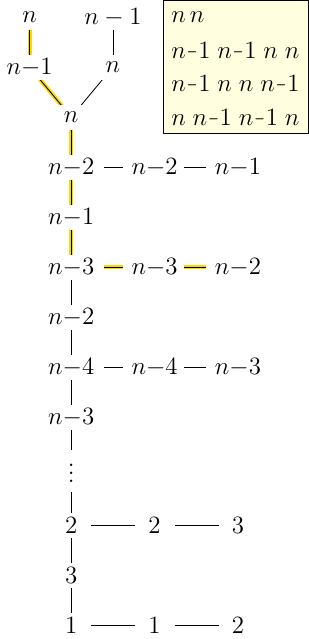}}
        \fbox{\includegraphics[align=t, scale=0.78]{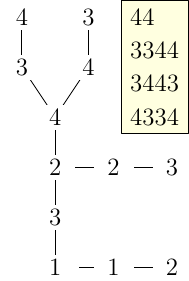}}
        \caption{The prefix-tree $\prefixTree{n}$ (left).
        It is drawn with its two roots ($n$ and $n-1$) at the top, and its $n-2$ leaves on the right.
        The tree's root-to-leaf paths encode annexes.
        For example, the highlighted path encodes the annex $\alpha' = n \, n{-}1 \, n \, n{-}2 \, n{-}1 \, n{-}3 \, n{-}3 \, n{-}2$.
        Thus, $\alpha' \beta' \in \avoid[2]{n}{\vincular{121}, 123, 132, 213}$ for any $\beta' \in \avoid[2]{n-4}{\vincular{121}, 123, 132, 213}$ by Lemma \ref{lem:composition}.
        Starting at the second copy of $n$ the vertical edges follow a \emph{two steps forward one step back} pattern downward, and horizontal edges complete the prefix with the two smallest and yet-to-be paired values.
        The inset shows the four short annexes (e.g., $n \; n$).
        The prefix-tree $\prefixTree{n}$ for $n=4$ (right).
        The four root-to-leaf paths give the annexes $434223$, $344223$, $43423112$, and $34423112$ found in (a).
        The remaining annexes are the special cases $44$ with $3344$, $3443$, and $4334$.
        }        
        \label{fig:proofc_forest}        
    \end{subfigure}
    \caption{Illustrating Theorem \ref{thm:c} and its proof.
    (a) The Fibonacci-squared words are $2$-regular words avoiding $\{\vincular{121}, 123, 132, 213\}$ and they are enumerated by the Fibonacci-squared numbers $c(n)$ (or equivalently, $c''(n)$). 
    Each word is written in its standard partition $\gamma = \alpha \beta$ where $\alpha$ is the annex and $\beta$ is the base.
    The bases are smaller Fibonacci-squared words, while each annex is one of four special cases or can be visualized as a labeled root-to-leaf path in the prefix-tree (b).
    }
    \label{fig:proofc}
\end{figure}

\section{Final Remarks}
\label{sec:final}

We considered three pattern avoiding results involving $k$-regular words and Fibonacci sequences.
This includes a simple proof of a known result for Fibonacci-$k$ words, a new result for $k$-Fibonacci words, and a new result for the Fibonacci-squared sequence.
We hope that these results, together with those for $k$-Catalan sequences (see Section \ref{sec:intro_regular}) and new results on $(1,b)$-based Pell numbers by \cite{hartung2024regular} will serve as inspiration for further study of pattern avoidance in regular words or even non-regular words (see \cite{burstein1998enumeration}).
We conclude with some additional comments and thoughts.

\subsection{Other Fibonacci Sequences with Parameter \texorpdfstring{$k$}{k}}
\label{sec:final_kk}

Pattern avoidance with $k$-regular words is particularly promising for sequence families with one parameter. 
Other Fibonacci sequences with this property include those using base cases $(1,k)$ or $(k,1)$ or $(k,k)$.
Another natural target is the $(1,1)$-based $k$-Fibonacci-$k$ sequence, which satisfies the following formula.
\begin{equation} \label{eq:d}
d_k(n) = k \cdot d_k(n-1) + k \cdot d_k(n-2) \text{ with } d_k(0) = 1 \text{ and } d_k(1) = 1.
\end{equation}
Interestingly, a pattern avoidance result for $d_2(n)$ was established using $d$-permutations by \cite{sun2022d}.
More specifically, $d_2(n)$ is the number of $3$-permutations of $[n]$ avoiding $231$ and $312$.
See \cite{bonichon2022baxter} for further information on $d$-permutations (which are not $d$-regular permutations).
These sequences are summarized in Table \ref{tab:ksequencesk}. Note $d_k(n)$ does not refer to the double list notion in Table \ref{tab:not}.

Another single parameter generalization is the \emph{$k$-generalized Fibonacci numbers} which involve summing the previous $k$ entries in the sequence.
Pattern avoidance results involving these numbers have been considered by \cite{egge2004231}. 


\begin{table}[ht]
\centering
\small
\begin{tabular}{|@{\;\;}r@{\,}c@{\,}r@{\;\;}c@{\;\;}l@{\;\;}r@{\;\;}|} \hline
\rowcolor{gray!20}
\multicolumn{3}{|c}{Recurrence $d(n)$} & Numbers & \multicolumn{1}{c}{\textsc{Oeis}}  & \multicolumn{1}{c@{\;\;}|}{Pattern Avoidance} \\ \hline
$d(n{-}1)$ &         $+$ & $d(n{-}2)$         & $1, 1, 2, 3, 5, 8, 13, 21, \ldots$ & \OEIS{A000045}$^{\dag}$ & $\avoid{n}{123,132,213}$ \\
$2 \cdot d(n{-}1)$ & $+$ & $2 \cdot d(n{-}2)$ & $1, 1, 4, 10, 28, 76, 208, 568, \ldots$ & \OEIS{A026150} & $S_n^{2}(231,312)$ \\
$3 \cdot d(n{-}1)$ & $+$ & $3 \cdot d(n{-}2)$ & $1, 1, 6, 21, 81, 306, 1161, 4401, \ldots$ & \OEIS{A134927} & \\
$4 \cdot d(n{-}1)$ & $+$ & $4 \cdot d(n{-}2)$ & $1, 1, 8, 36, 176, 848, 4096, 19776, \ldots$ & \OEIS{A164545}$^{\ddag}$ & \\ 
$5 \cdot d(n{-}1)$ & $+$ & $5 \cdot d(n{-}2)$ & $1, 1, 10, 55, 325, 1900, 11125, 65125, \ldots$ & \OEIS{A188168} & \\ 
$6 \cdot d(n{-}1)$ & $+$ & $6 \cdot d(n{-}2)$ & $1, 1, 12, 78, 540, 3708, 25488, 175176, \ldots$ &      & \\ 
$7 \cdot d(n{-}1)$ & $+$ & $7 \cdot d(n{-}2)$ & $1, 1, 14, 105, 833, 6566, 51793, 408513, \ldots$ &      & \\
$8 \cdot d(n{-}1)$ & $+$ & $8 \cdot d(n{-}2)$ & $1, 1, 16, 136, 1216, 10816, 96256, 856576, \ldots$ &      & \\
$9 \cdot d(n{-}1)$ & $+$ & $9 \cdot d(n{-}2)$ & $1, 1, 18, 171, 1701, 16848, 166941, 1654101, \ldots$ &      & \\
\hline
\end{tabular}
\caption{The $k$-Fibonacci-$k$ sequences $d_k(n) = k \cdot d_k(n-1) + k \cdot d_k(n-2)$ with $d_k(0) = 1$ and $d_k(1) = 1$ provide an alternate way of parameterizing the Fibonacci numbers.  
A previous pattern-avoiding result for $k=2$ is by \cite{sun2022d} where $S_n^{2}$ denotes pattern avoidance in $3$-permutations.
(Note that a $3$-permutation is not a $3$-regular permutation but rather a pair of permutations and their $3$-dimensional diagram \cite{bonichon2022baxter}.)
$^{\dag}$\textsc{Oeis} entry starts with~$0$.
$^{\ddag}$\textsc{Oeis} entry omits an initial~$1$.}
\label{tab:ksequencesk}
\end{table}

\subsection{Avoiding Patterns and Multi-Patterns}
\label{sec:final_alphaBeta}

When working with regular words it is natural to avoid \emph{multi-patterns} (i.e., patterns that have repeated symbols).
In particular, our three main theorems each included the avoidance of one multi-pattern.
More broadly, we can ask for regular word avoidance results for any subset of patterns and multi-patterns.
To aid in this discussion, let $\MULTIPATS[m]{n}$ be the set of strings of length $n$ containing each symbol in $[m]$ at least once.
For example, $\MULTIPATS[2]{3} = \{112, 121, 122, 211, 212, 221\}$ is the set of multi-patterns with $2$ distinct symbols and length $3$.
Standard patterns arise when $m=n$ (i.e., $\MULTIPATS[n]{n} = \PERMS{n}$).

One challenge in this direction of inquiry is the sheer number of possibilities.
For example, if we consider standard patterns and multi-patterns of length $n=3$, then there are $2^{12} = 4096$ different subsets of $\PERMS{3} \cup \MULTIPATS[2]{3}$.
Moreover, each of these subsets must be considered for regular words of various frequencies.
For this reason we must be conservative when trying to summarize all possible results of this type, even when restricted to patterns of length three.

One manageable goal is to understand all possibilities of avoiding one standard pattern $\alpha \in \PERMS{3}$ and one multi-pattern $\beta \in \MULTIPATS[2]{3}$.
There are $|\PERMS{3}| \cdot |\MULTIPATS[2]{3}|/4 = \frac{6 \cdot 6}{4} = 9$ distinct $(\alpha, \beta)$ pairs to considewr up to the symmetries of the square. 
Table \ref{tab:alphaBeta} summarizes the known results, including our new result in Theorem \ref{thm:b}.

When creating Table \ref{tab:alphaBeta} we observed that there are two non-isomorphic case that are currently unresolved.
Based on computational experiments involving generating and testing regular words (see \cite{williams2009loopless}), we feel comfortable in stating the following conjecture for the missing cases.
While formula \eqref{eq:123_121} in Conjecture \ref{con:123_121_and_132_212} is fairly complicated, Table \ref{tab:alphaBeta} shows that there are existing \textsc{Oeis} entries for $r=2,3,4,5$.

\begin{con} \label{con:123_121_and_132_212}
The number of $r$-regular words over $[n]$ that avoid $123$ and $121$, or $132$ and $212$, is below.
\begin{equation} \label{eq:123_121}
    |\avoid[r]{n}{123, 121}| = |\avoid[r]{n}{132, 212}| = \sum\limits_{i=0}^{n}\frac{\binom{n}{i}}{n{-}i{+}1} \binom{n{+}(r{-}1) \cdot i{-}1}{n{-}i}
\end{equation}
\end{con}



\begin{table}
\scriptsize
\centering
\begin{tabular}{|@{\,\;\;}c@{\;}c@{\;\;\,}c@{\;\;}c@{\;}c@{\;}c@{\;}c@{\;}c@{\;}c@{\;}c@{\,\;\;}|} \hline
\rowcolor{gray!20}
\multicolumn{1}{|c}{$\alpha,\beta$} & Parameterized Sequence & Proof & $r=2$ & $r=3$ & $r=4$ & $r=5$ & $r=6$ & $r=7$ & $r=8$  \\ \hline 
\makecell{$123,112$ \\[-0.2em] $132,121$ \\[-0.2em] $132,122$} & $(r+1)$-Catalan &  \makecell{\cite{williams2023pattern} \\[-0.2em] \cite{kuba2012enumeration} \\[-0.2em] \cite{defant2020stack}} & \OEIS{A001764} & \OEIS{A002293} & \OEIS{A002294} & \OEIS{A002295} & \OEIS{A002296} & \OEIS{A007556} & \OEIS{A059967} \\ \hline 
$132,112$ & $r$-Fibonacci & Theorem \ref{thm:b} & \OEIS{A001333} & \OEIS{A003688} & \OEIS{A015448} & \OEIS{A015449} & \OEIS{A015451} & \OEIS{A015453} & \OEIS{A015454} \\ \hline 
$132,221$ & $(1,r+1)$-based Pell & \cite{hartung2024regular} & \OEIS{A001333} & \OEIS{A048654} & \OEIS{A048655} & \OEIS{A048693} & \OEIS{A048694} & \OEIS{A048695} & \OEIS{A048696} \\ \hline 
$132,211$ & $r(n-1) + 1$ & \cite{hartung2024regular} & \OEIS{A005408} & \OEIS{A016777} & \OEIS{A016813} & \OEIS{A016861} & \OEIS{A016921} & \OEIS{A016993} & \OEIS{A017077} \\ \hline 
$123,211$ & $1,r+1,0,0,0,\ldots$ & \cite{hartung2024regular} & & & & & & & \\ \hline 
\makecell{$123,121$ \\ $132,212$} & {\tiny $\displaystyle \sum\limits_{i=0}^{n}\frac{\binom{n}{i}}{n{-}i{+}1} \binom{n{+}(r{-}1) \cdot i{-}1}{n{-}i}$} & Conjecture \ref{con:123_121_and_132_212} & \OEIS{A109081} & \OEIS{A161797} & \OEIS{A321798} & \OEIS{A321799} & & & \\ \hline 
\end{tabular}
\caption{Avoiding $\alpha \in \PERMS{3}$ and $\beta \in \MULTIPATS[2]{3}$~in $r$-regular words:
known results, new results, and a conjecture.
The $|\PERMS{3}| \cdot |\MULTIPATS[2]{3}|=36$ pairs of $(\alpha,\beta)$ partition into the $\frac{36}{4}=9$ equivalence classes listed above based on the symmetries of the square (e.g., $\{(123,121), (321,121), (123,212), (321,212)\}$ is represented by $(123,121)$ above).
}
\label{tab:alphaBeta}
\end{table}

\acknowledgements

We'd like to thank the referees assisting in the preparation of this document.
This includes encouragement to prove Theorem \ref{thm:c}, which was initially a conjecture.
We would also like to thank the \textsc{Oeis}.

\nocite{*}
\bibliographystyle{abbrvnat}
\bibliography{refs}

@misc{OEIS,
    Author = {{OEIS Foundation Inc.}},
    Note = {Published electronically at \url{http://oeis.org}},
    Title = {The {O}n-{L}ine {E}ncyclopedia of {I}nteger {S}equences},
    Year = {2025}
}

@article{simion1985restricted,
  title={Restricted permutations},
  author={Simion, Rodica and Schmidt, Frank W},
  journal={European Journal of Combinatorics},
  volume={6},
  number={4},
  pages={383--406},
  year={1985},
  publisher={Elsevier}
}

@article{baril2019enumeration,
  title={{Enumeration of {\L}ukasiewicz paths modulo some patterns}},
  author={Baril, Jean-Luc and Kirgizov, Sergey and Petrossian, Armen},
  journal={Discrete Mathematics},
  volume={342},
  number={4},
  pages={997--1005},
  year={2019},
  publisher={Elsevier}
}

@article{mansour2002refined,
  title={Refined restricted permutations avoiding subsets of patterns of length three},
  author={Mansour, Toufik and Robertson, Aaron},
  journal={Annals of Combinatorics},
  volume={6},
  number={3},
  pages={407--418},
  year={2002},
  publisher={Springer}
}

@article{archer2018pattern,
  title={Pattern restricted quasi-{S}tirling permutations},
  author={Archer, Kassie and Gregory, Adam and Pennington, Bryan and Slayden, Stephanie},
  journal={arXiv preprint arXiv:1804.07267},
  year={2018}
}

@article{elizalde2024pattern,
  title={Pattern avoidance in nonnesting permutations},
  author={Elizalde, Sergi and Luo, Amya},
  journal={arXiv preprint arXiv:2412.00336},
  year={2024}
}

@article{falcon2007fibonacci,
  title={{On the Fibonacci k-numbers}},
  author={Falc{\'o}n, Sergio and Plaza, {\'A}ngel},
  journal={Chaos, Solitons \& Fractals},
  volume={32},
  number={5},
  pages={1615--1624},
  year={2007},
  publisher={Elsevier}
}

@article{gupta2012generalized,
  title={{Generalized Fibonacci Sequences}},
  author={Gupta, VK and Panwar, Yashwant K and Sikhwal, Omprakash},
  journal={Theoretical Mathematics and Applications},
  volume={2},
  number={2},
  pages={115--124},
  year={2012}
}

@article{panwar2021note,
  title={{A note on the generalized k-Fibonacci sequence}},
  author={Panwar, Yashwant},
  journal={NATURENGS},
  volume={2},
  number={2},
  pages={29--39},
  year={2021},
  publisher={Malatya Turgut Ozal University}
}

@article{baril2011classical,
  title={Classical sequences revisited with permutations avoiding dotted pattern},
  author={Baril, Jean-Luc},
  journal={{The Electronic Journal of Combinatorics}},
  volume={18},
  number={1},
  year={2011}
}

@inproceedings{gao2014tackling,
  title={Tackling sequences from prudent self-avoiding walks},
  author={Gao, Shanzhen and Chen, Keh-Hsun},
  booktitle={Proceedings of the International Conference on Foundations of Computer Science (FCS)},
  pages={1},
  year={2014},
  organization={The Steering Committee of The World Congress in Computer Science, Computer Engineering and Applied Computing (WorldComp)},
}

@unpublished{gao2011sequences,
  title={Sequences arising from prudent self-avoiding walks},
  author={Gao, S and Niederhausen, H},
  note={Unpublished manuscript available at: \url{http://world-comp.org/preproc2014/FCS2696.pdf}},
  year={2011},
}

@article{cratty2016pattern,
  title={Pattern avoidance in double lists},
  author={Cratty, Charles and Erickson, Samuel and Negassi, Frehiwet and Pudwell, Lara},
  journal={Involve, a Journal of Mathematics},
  volume={10},
  number={3},
  pages={379--398},
  year={2016},
  publisher={Mathematical Sciences Publishers}
}

@article{barcucci2006fibonacci,
  author={E. Barcucci and A. Bernini and M. Poneti},
  title={{From Fibonacci to Catalan permutations}},
  journal={arXiv preprint math/0612277},
  year={2006}
}

@article{sun2022d,
  title={On $ d $-permutations and Pattern Avoidance Classes},
  author={Sun, Nathan},
  journal={Annals of Combinatorics},
  volume={28},
  number={3},
  pages={701-732},
  year={2024}
}

@article{bonichon2022baxter,
  title={Baxter $d$-Permutations and Other Pattern-Avoiding Classes},
  author={Bonichon, Nicolas and Morel, Pierre-Jean},
  journal={Journal of Integer Sequences},
  volume={25},
  number={22.8.3},
  pages={},
  year={2022}
}

@article{bevan2015permutation,
  title={Permutation patterns: basic definitions and notation},
  author={Bevan, David},
  journal={arXiv preprint arXiv:1506.06673},
  year={2015}
}

@article{defant2020stack,
  title={Stack-sorting for words},
  author={Defant, Colin and Kravitz, Noah},
  journal={Australasian Journal of Combinatorics},
  volume={77},
  number={1},
  pages={51--68},
  year={2020}
}

@inproceedings{williams2023pattern,
  title={{Pattern Avoidance for $k$-Catalan Sequences}},
  author={Williams, Aaron},
  booktitle={Proceedings of the 21st International Conference on Permutation Patterns},
  pages={147--149},
  year={2023}
}

@inproceedings{hartung2024regular,
  title={{Regular Word Pattern Avoidance for $(1,b)$-Based Pell Numbers}},
  author={Hartung, Elizabeth and Williams, Aaron},
  booktitle={Proceedings of the 22nd International Conference on Permutation Patterns},
  pages={30--33},
  year={2024}
}

@article{knuth1968art,
  title={{The Art of Computer Programming, Vol 1: Fundamental Algorithms}},
  xnote={Section 2.2.1, Exercises 4 and 5},
  author={Knuth, Donald E},
  journal={Algorithms. Reading, MA: Addison-Wesley},
  year={1968}
}

@book{macmahon1915combinatory1,
  title={Combinatory Analysis},
  author={MacMahon, Percy Alexander},
  volume={1},
  year={1915},
  publisher={Cambridge University Press}
}

@article{gessel1978stirling,
  title={Stirling polynomials},
  author={Gessel, Ira and Stanley, Richard P},
  journal={Journal of Combinatorial Theory, Series A},
  volume={24},
  number={1},
  pages={24--33},
  year={1978},
  publisher={Elsevier}
}

@article{park1994r,
  title={The r-multipermutations},
  author={Park, SeungKyung},
  journal={Journal of Combinatorial Theory, Series A},
  volume={67},
  number={1},
  pages={44--71},
  year={1994},
  publisher={Elsevier}
}

@article{janson2011generalized,
  title={{Generalized Stirling permutations, families of increasing trees and urn models}},
  author={Janson, Svante and Kuba, Markus and Panholzer, Alois},
  journal={Journal of Combinatorial Theory, Series A},
  volume={118},
  number={1},
  pages={94--114},
  year={2011},
  publisher={Elsevier}
}

@article{kuba2012enumeration,
  title={Enumeration formulae for pattern restricted {S}tirling permutations},
  author={Kuba, Markus and Panholzer, Alois},
  journal={Discrete Mathematics},
  volume={312},
  number={21},
  pages={3179--3194},
  year={2012},
  publisher={Elsevier}
}

@article{kuba2019stirling,
  title={Stirling permutations containing a single pattern of length three.},
  author={Kuba, Markus and Panholzer, Alois},
  journal={Australasian Journal of Combinatorics},
  volume={74},
  pages={216--239},
  year={2019}
}

@inproceedings{shen2021k,
  title={A k-ary middle levels conjecture},
  author={Shen, Xi Sisi and Williams, Aaron},
  booktitle={Proceedings of the 23rd Thailand-Japan Conference on Discrete and Computational Geometry, Graphs, and Games},
  year={2021}
}

@article{mutze2016proof,
  title={Proof of the middle levels conjecture},
  author={M{\"u}tze, Torsten},
  journal={Proceedings of the London Mathematical Society},
  volume={112},
  number={4},
  pages={677--713},
  year={2016},
  publisher={Oxford University Press}
}

@article{mutze2023book,
  title={A book proof of the middle levels theorem},
  author={M{\"u}tze, Torsten},
  journal={Combinatorica},
  howpublished = {\url{https://doi.org/10.1007/s00493-023-00070-3}},
  year={2023}
}

@article{gregor2023star,
  title={{Star transposition Gray codes for multiset permutations}},
  author={Gregor, Petr and Merino, Arturo and M{\"u}tze, Torsten},
  journal={Journal of Graph Theory},
  volume={103},
  number={2},
  pages={212--270},
  year={2023},
  publisher={Wiley Online Library}
}

@inproceedings{williams2009loopless,
  title={Loopless generation of multiset permutations using a constant number of variables by prefix shifts},
  author={Williams, Aaron},
  booktitle={Proceedings of the twentieth annual ACM-SIAM symposium on discrete algorithms},
  pages={987--996},
  year={2009},
  organization={SIAM}
}

@article{babson2000generalized,
  title={Generalized permutation patterns and a classification of the {M}ahonian statistics.},
  author={Babson, Eric and Steingr{\'\i}msson, Einar},
  journal={S{\'e}minaire Lotharingien de Combinatoire [electronic only]},
  volume={44},
  pages={B44b--18},
  year={2000},
  publisher={Universit{\"a}t Wien, Fakult{\"a}t f{\"u}r Mathematik}
}

@article{pons2015lattice,
  title={A lattice on decreasing trees: the metasylvester lattice},
  author={Pons, Viviane},
  journal={Discrete Mathematics \& Theoretical Computer Science},
  number={Proceedings},
  year={2015},
  publisher={Episciences. org}
}

@article{ceballos2024weak,
  title={The $s$-Weak Order and $s$-Permutahedra {I}: Combinatorics and Lattice Structure},
  author={Ceballos, Cesar and Pons, Viviane},
  journal={SIAM Journal on Discrete Mathematics},
  volume={38},
  number={4},
  pages={2855--2895},
  year={2024},
  publisher={SIAM}
}

@article{bergeron2012higher,
  title={Higher trivariate diagonal harmonics via generalized Tamari posets},
  author={Bergeron, Fran{\c{c}}ois and Pr{\'e}ville-Ratelle, Louis-Fran{\c{c}}ois},
  journal={Journal of Combinatorics},
  volume={3},
  number={3},
  pages={317--341},
  year={2012}
}

@article{hivert2005algebra,
  title={The algebra of binary search trees},
  author={Hivert, Florent and Novelli, J-C and Thibon, J-Y},
  journal={Theoretical Computer Science},
  volume={339},
  number={1},
  pages={129--165},
  year={2005},
  publisher={Elsevier}
}

@article{novelli2020hopf,
  title={Hopf algebras of m-permutations,(m+ 1)-ary trees, and m-parking functions},
  author={Novelli, Jean-Christophe and Thibon, Jean-Yves},
  journal={Advances in Applied Mathematics},
  volume={117},
  pages={102019},
  year={2020},
  publisher={Elsevier}
}

@phdthesis{burstein1998enumeration,
  title={Enumeration of words with forbidden patterns},
  author={Burstein, Alexander},
  year={1998},
  school={University of Pennsylvania}
}

@article{egge2004231,
  title={{231-Avoiding involutions and Fibonacci numbers}},
  author={Egge, Eric S and Mansour, Toufik},
  journal={Australasian Journal of Combinatorics},
  volume={30},
  pages={75--84},
  year={2004}
}
\label{sec:biblio}

\end{document}